\documentclass[12pt]{amsart}
\usepackage[colorlinks=true,citecolor=blue,linkcolor=magenta]{hyperref}
\usepackage{amsmath,mathtools}
\usepackage{verbatim}
\usepackage{amsfonts}
\usepackage{amssymb}
\usepackage{color,colortbl}
\usepackage{enumerate}
\usepackage[top=1in, bottom=1in, left=1in, right=1in]{geometry}
\usepackage{mathrsfs}
\usepackage{blkarray}
                   
\usepackage{stmaryrd}
\usepackage{bbold}

\usepackage{cleveref}
\usepackage{soul}
\usepackage{mathrsfs}
\usepackage{tikz-cd}
\usepackage{tikz}
\usepackage{mathtools}
\theoremstyle{definition}
\newtheorem{theorem}{Theorem}[section]
\newtheorem{theoremx}{Theorem}
\newtheorem*{theorem*}{Theorem}

\numberwithin{equation}{section}

\newtheorem{corollary}[theorem]{Corollary}
\newtheorem{lemma}[theorem]{Lemma}
\newtheorem{proposition}[theorem]{Proposition}

\theoremstyle{definition}
\newtheorem{definition}[theorem]{Definition}

\newtheorem{example}[theorem]{Example}

\newtheorem*{conjecture*}{Berger's Conjecture}
\newtheorem{remark}[theorem]{Remark}

\newtheoremstyle{TheoremNum}
{8pt}{8pt}              
{\upshape}                      
{}                              
{\bfseries}                     
{.}                             
{.5em}                             
{\theoremname{#1}\theoremnote{ \bfseries #3}}
\theoremstyle{TheoremNum}

\newcommand{\m}{\mathfrak{m}}
\newcommand{\n}{\mathfrak{n}}

\renewcommand{\(}{\left(}
\renewcommand{\)}{\right)}

\newcommand{\CC}{\mathbb{C}}

\newcommand{\cC}{\mathfrak{C}}

\newcommand{\Rank}{\operatorname{rank}}

\newcommand{\Hom}{\operatorname{Hom}}

\newcommand{\Tor}{\operatorname{Tor}}

\newcommand{\Ann}{\operatorname{Ann}}

\newcommand{\depth}{\operatorname{depth}}

\newcommand{\h}{\operatorname{h}}

\newcommand{\hun}{\displaystyle \h(\Omega_{R/k})}

\renewcommand{\leq}{\leqslant}
\renewcommand{\geq}{\geqslant}
\newcommand{\ds}{\displaystyle}

\newcommand{\tens}{\otimes}

\DeclareMathOperator{\chr}{char}

\title{Partial Trace Ideals And Berger's Conjecture}
\author{Sarasij Maitra}
\date{}
\email{sm3vg@virginia.edu}
\address{University of Virginia, Charlottesville, VA}
\begin{document}
	\begin{abstract}
		For any finitely generated module $M$ with non-zero rank over a commutative one-dimensional Noetherian local domain, we study a numerical invariant $\h(M)$ based on a partial trace ideal of $M$.
		We study its properties and explore relations between this invariant and the colength of the conductor. Finally we apply this to the universally finite module of differentials $\Omega_{R/k}$, where $R$ is a complete $k$-algebra with $k$ any perfect field, to study a long-standing conjecture due to R. W. Berger.  
	\end{abstract}
	
\maketitle

\section*{Introduction}	Let $k$ be a perfect field and let $R$ be
a reduced local $k$-algebra of dimension one. There is a long standing conjecture of R.W. Berger that $R$ is regular if and only if the universally finite differential module $\ds \Omega_{ R/k}$ is torsion-free. The difficult part is to show that for a singular $R$, there is a nonzero torsion element in $\Omega_{ R/k}$. 

A lot of results, all supporting the conjecture, exist in the literature.  They use a variety of techniques involving multiplicity, linkage, embedding dimension, deviations, smoothability, maximal torsion, equisingularity, quadratic transforms and Artinian methods. The reader can refer to numerous sources like  \cite{MR152546}, \cite{scheja1970differentialmoduln}, \cite{MR470245}, \cite{MR512469}, \cite{MR571575}, \cite{MR629286}, \cite{MR754335}, \cite{Herzog1984DifferentialsOL},  \cite{MR836868}, \cite{MR948266}, \cite{pohl1989torsion1}, \cite{pohl1989torsion}, \cite{Guttes1990}, \cite{MR1035346}, \cite{MR1129516}, \cite{MR1129517}, \cite{MR1637796} and \cite{MR1679664}. A very nice summary of most of these results, along with the main ideas of proofs, can be found in \cite{berger1994report}.

  In this paper, we focus on the case when $R$ is a complete local domain. Our approach is to introduce an invariant $\h(M)$ for any finitely generated module $M$ with a non-zero rank over a one-dimensional local Noetherian domain. It is defined as follows: 
  \[\h(M):=\min \{\lambda(R/J)~|~\exists~\text{homomorphism~}f:M\to R\text{~such that~}J=f(M)\}\] where $\lambda(\cdot)$ denotes length as an $R$-module.  We refer to any such $J$ that achieves the above minimum as a \textit{partial trace ideal of $M$}. For any ideal $I$, we call $\lambda(R/I)$  the colength of $I$. Thus, a partial trace ideal of $M$ is one which achieves the minimum colength among all the images of $M$ in $R$.
  
  We study various properties of this invariant and explore the relationship with the conductor ideal under suitable hypothesis. Finally we apply this to $M=\Omega_{ R/k}$ and study the conjecture in terms of  bounds of this invariant. 
  We give a brief discussion (see \Cref{h=0}) about the case when the invariant takes the value $0$. In this case, the conjecture of interest becomes quite easy to analyze. The major points of discussion, thus, will be the instances when $\hun$ is positive. The motivation for the definition came from Scheja's proof of  Berger's conjecture in the graded setup. One of our main results is the following (see \Cref{mainthm}):
\begin{theoremx}\label{MAINTHM}
Let	$\ds (R,\m ,k)$ be a Noetherian complete local one-dimensional domain such that $\ds R=S/I$ where $S=k[[X_1,\dots,X_n]]$, $k$ is perfect and $n\geq 2$. Let $\n$ be the maximal ideal of $S$ and $\ds I\subseteq \n^{s+1}$ for $s\geq 1$. If  $\ds \h(\Omega_{R/k})< {{n+s\choose s}\frac{s}{s+1}}$, then $\ds  \Omega_{ R/k}$ has torsion. In particular, Berger's Conjecture is true whenever $\hun$ satisfies the stated bound.
\end{theoremx}

This provides a generalization of a proof due to G. Scheja (see \Cref{corScheja}). We also find a relationship between $\hun$ and the colength of the conductor. The main idea of the proof here goes through studying colengths of isomorphic copies of ideals. We establish an equivalent condition for an \textit{optimal colength} to be achieved  under suitable hypothesis (see \Cref{prop.propo1}).

\begin{theoremx}\label{MAINTHM2}
	Let $\ds (R,\m,k)$ be a one-dimensional Noetherian local domain with integral closure $\overline{R}$ and fraction field $K$. Further assume that $\ds \widehat{R}$ is reduced. For any ideal $J$ of $R$,  consider the following statements:
	\begin{itemize}
		\item[$(a)$] $\ds \h(J)=\lambda\({R}/{J}\)$.
		
		\item[$(b)$] $\ds R:_KJ\subseteq \overline{R}.$
	\end{itemize}  
	Then $(b)$ implies $(a)$. Further if $\ds \overline{R}$ is a $DVR$, then $(a)$ implies $(b)$.
	
\end{theoremx} 
\noindent In statement $(b)$ above, $R:_KJ:=\{\alpha\in K\mid \alpha J\subseteq R\}$. Whether $(a)$ implies $(b)$ without the $DVR$ assumption may be interesting to investigate. Since in this article we eventually apply our results to the setup of Berger's Conjecture,  $\overline{R}$ being a $DVR$ is enough to proceed.

The statements in \Cref{MAINTHM2} in fact suggest a link to the conductor ideal $\cC$ which is the largest common ideal between $R$ and $\overline{R}$ (see \Cref{prop.cor1}, \Cref{prop.thm1}).
  
  We also provide an explicit way of computing $\hun$ under some additional hypothesis, which can be efficiently implemented on Macaulay 2 to check the applicability of \Cref{MAINTHM} given an instance of the conjecture (see \Cref{explicitcomp}, \Cref{hcomprem2} and  \Cref{explicitcompcor}). Here we use $v$ to denote the discrete valuation on $K=k((t))$ which takes a power series in $t$ as input and outputs the least exponent of $t$ in the series. We call it the \textit{valuation function} on  $\bar{R}=k[[t]]$. The hypothesis guarantees that $\overline{R}$ and $K$ can be written in the forms indicated above where $t$ is the chosen uniformizing parameter.
  
  \begin{theoremx}\label{MAINTHM4}
 	Let	$\ds (R,\m ,k)$ be a Noetherian complete local one-dimensional domain such that $\ds R=S/I$ where $S=k[[X_1,\dots,X_n]]$, $k$ is algebraically closed of characteristic $0$, $n\geq 2$ and let $K$ be the fraction field of $R$. Assume $\ds I\subseteq (X_1,\dots,X_n)^{s+1}$ for $s\geq 1$. Suppose $\overline{R}=k[[t]]$ with valuation function $v$, which is a discrete valuation on $K$. Write $x_i(t)$ to denote the image of $X_i$ in $\overline{R}$ and set $x_i'(t)=\frac{d}{dt}x_i(t)$.  Finally let $\mathcal{D}=Rx_1'(t)+\dots +Rx_n'(t)$ and $v(\mathcal{D}^{-1})=\min\{v(\alpha)~|~\alpha\in K, \alpha \mathcal{D}\subseteq R\}$. If $$v(\mathcal{D}^{-1})< {n+s\choose s}\frac{s}{s+1}+\lambda(\overline{R}/R)-\lambda(\overline{R}/\mathcal{D}),$$\text{~then} $\Omega_{ R/k}$ has torsion. So, Berger's Conjecture is true in this case.
 \end{theoremx}
\noindent In fact, the expression $\lambda(\overline{R}/R)-\lambda(\overline{R}/\mathcal{D})$ is always non-negative and can be computed by \textit{counting the number of missing valuations from the value semi-groups of $R$ and $\mathcal{D}$ respectively} (\cite[Proposition 2.9]{herzog1971wertehalbgruppe}). There exist significant discussions in the literature on the case when this quantity is equal to zero: this is referred to as \textit{maximal torsion}. Berger's conjecture is resolved in this situation (see \cite[2.2]{berger1994report}).

The structure of this paper is as follows.  In \Cref{invariant}, we define the invariant $\h(\cdot)$ and study important properties of it. We  establish the relationship with the conductor ideal by proving \Cref{MAINTHM2} and \Cref{prop.thm1}. In \Cref{sec_traceideal}, we relate the invariant to the trace ideal of the module under consideration and provide a lower bound on $\h(\cdot)$ under suitable hypothesis. We also show that a regular trace ideal always achieves its minimal colength among all its isomorphic copies (see \Cref{trace_prop}). In \Cref{invariantapp}, we study $\hun$ in detail and prove \Cref{MAINTHM}. In \Cref{explicitsec}, we prove \Cref{MAINTHM4} and also illustrate the computations using examples (\Cref{example2}, \Cref{example3}, \Cref{example4}). In particular, \Cref{example2} shows that this approach can lead to settling the conjecture for certain classes of rings (also, see \Cref{maincor}). We recover G. Scheja's proof for instance (see \Cref{corScheja}).  We conclude
by mentioning a few further questions.

\section{Preliminaries}\label{prelim}
Throughout the paper, $R$ will be a commutative local ring with unity. By $(R,\m,k)$, we denote a Noetherian local ring with maximal ideal $\m$ and residue field $k$. Let $\widehat{R}$ denote the $\m$-adic completion of $R$. For any module $M$, we will denote by $\lambda(M)$, the length of  $M$ (whenever length makes sense) and $\mu(M)$ will denote the minimal number of generators of $M$. All modules $M$ considered here will be finitely generated. The Hilbert-Samuel multiplicity of $M$ with respect to any $\m-$primary ideal ${I}$ will be denoted by $e({I};M)$. A module $M$ is Maximal Cohen Macaulay (MCM) if $\depth(M)=\dim(M)=\dim(R)$.

For the majority of this paper we shall only require $(R,\m,k)$ to be a one-dimensional local Noetherian domain with $\widehat{R}$ reduced, which guarantees that the integral closure $\bar{R}$ is finite over $R$. Also, because $R$ is a domain, the rank of any module $M$ is simply the dimension of $M\tens_R K$ over $K$ where $K$ is the fraction field of $R$. 

Since we are interested in the question of Berger's Conjecture, the prototypical example that the reader can keep in mind throughout is the following: Let $k$ be a perfect field.  Assume that $(R,\m,k)$ is a $k$-algebra which is a complete local one-dimensional domain of embedding dimension $n$, i.e., $\mu(\m)=n$. By the Cohen structure theorem, $\ds R=k{[[X_1,X_2,\dots,X_n]]}/{I}$ for some prime ideal $I$ in $\ds S=k[[X_1,\dots,X_n]]$ with $I\subseteq (X_1,\dots, X_n)^2$.
Let $K$ be the fraction field of $R$ and $\ds \overline{R}$ be the integral closure of $R$ in $K$.

Let $\omega$ denote the canonical module of $R$ (which exists if we assume that $R$ is the quotient of a regular local ring as above, and can be identified with some ideal of $R$ \cite[Proposition 3.3.18]{bruns_herzog_1998}). For a detailed exposition on the various important properties of $\omega$, we refer the reader to \cite[Chapter 3]{bruns_herzog_1998}.

Our object of interest is the module of differentials of $R$. For a detailed description of this module, we refer the reader to \cite{MR864975}. In particular, the `universally finite module of differentials' and the module of differentials are not the same in  general. However, for the purposes of this paper, we can work with the following definition.

\begin{definition}\label{defOmega}
Let $R=S/I$ where $S=k[[X_1,\dots,X_n]]$ or $S=k[X_1,\dots,X_n]$ where $k$ is any field. The universally finite module of differentials of $R$, denoted $\ds \Omega_{ R/k}$, is the finitely generated $R$-module which has the following presentation: 
$$R^{\mu(I)}\xrightarrow{A} R^n\to \Omega_{ R/k}\to 0$$ where the matrix $A$ over $R$ is given as follows: if $I=(f_1,\dots,f_{\mu(I)})$, then  $ A=\({\frac{\partial f_j}{\partial X_i}}\)_{1\leq j\leq \mu(I)\atop 1\leq i\leq n}$.
\end{definition}

 Taking into account  the formal partial derivations $dX_i$, we can also describe $\Omega_{ R/k}$ as follows: if $\ds I=(f_1,\dots,f_m)$, then $ \Omega_{R/k}=\frac{\bigoplus_{i=1}^nRdX_i}{U}$ where $U$ is the $R$-module generated by the elements $ df_i= \sum\limits_{i=1}^n\frac{\partial f_j}{\partial X_i}dX_i, j=1,\dots,m$.

\begin{remark}\label{rankomega}
When $k$ is a perfect field, the following results are known.
\begin{itemize} 
\item[a)] $\ds \mu(\Omega_{ R/k})=n$ \cite[Corollary 6.5]{MR864975};
\item[b)] $\Rank_R(\Omega_{ R/k})=1$ (follows from Corollary 4.22, Proposition 5.7b and Theorem 5.10c, \cite{MR864975}).
\item[c)] There exists an exact sequence $ \ds I/I^{2}\to R^n\to \Omega_{ R/k}\to 0$ \cite[Theorem 1.5]{herzog1994module}. Let $\Gamma/I^2$ be the kernel of the first map. Tensoring the sequence $\ds 0\to I/\Gamma\to R^n\to \Omega_{ R/k}\to 0$  with $k$ and looking at the long exact sequence, we immediately obtain $\ds \dim_{k}\Tor^R_1(\Omega_{ R/k},k)\leq \mu(I)$. 
\end{itemize}
\end{remark}
Let $\ds \tau(\Omega_{R/k})$ denote the torsion submodule of $\Omega_{ R/k}$. We now explicitly state the Berger's conjecture in the  case when $R$ is a domain.

\begin{conjecture*}
Let $k$ be a perfect field and let $R$ be a complete one-dimensional local $k$-algebra which is a domain. Then $R$ is regular if and only if $\ds \tau(\Omega_{R/k})=0$.
\end{conjecture*}

	The more general statement of the conjecture only requires $R$ to be reduced. For the  general version, see the introduction in \cite{berger1994report}.

From \Cref{defOmega}, it is clear that if $R$ is regular, i.e., $n=1$, then $\ds \Omega_{R/k}$ is free. So $\tau(\Omega_{R/k})=0$. It is the converse direction that we need to study. Hence, from now on we assume that $n\geq 2$. Thus we shall be concerned mainly with the following situation:
			 $k$ is a perfect field, $\ds R=S/I$ where $S=k[[X_1,\dots,X_n]]$, $n\geq 2
			$,  $I$ is a prime ideal in $S$ with $I\subseteq \n^2$ where $\n$ is the maximal ideal of $S$ and $\text{height}(I)=n-1$.

Recall that for	a local ring $(R,\m,k)$, if  $E$ is the injective hull of $k$,  then for any module $M$, the \textit{Matlis Dual} of $M$ is defined as $\ds M^{\vee}:=\Hom_R(M,E).$ 
We refer the reader to \cite[Section 3.2]{bruns_herzog_1998} for more information on injective hulls and Matlis duals.

	 As we proceed through the paper, in every section we will set up additional notations and conventions whenever necessary and also make all the hypothesis explicit for convenience.

\section{An Invariant $\h(\cdot)$ and Some Properties}\label{invariant}

As mentioned in the Introduction, we shall study the conjecture by studying an invariant. The purpose of this section is to define this invariant and study some of its properties. 
Recall that  $\ds \lambda(M)$ denotes the length of an $R$-module $M$. Throughout this section, $(R,\m,k)$ denotes a one-dimensional Noetherian local ring. We shall explicitly mention any additional hypothesis. 

\begin{definition}\label{definv} Let $R$ be a local Noetherian one-dimensional domain. For any finitely generated $R$-module $M$ with a non-zero rank, let
	$$\h(M):=\min\{\lambda(R/J)~|~\text{there exists a homomorphism~} \phi:M\to R~\text{such that~}\phi(M)=J\}.$$
	\noindent We say that an ideal \textbf{\textit{$\ds J$ realizes $M$}}  if $M$ surjects to $J$ and $\ds \h(M)=\lambda(R/J)$.
\end{definition}

The assumption that $R$ is a one-dimensional domain ensures that $\ds \h(M)$ is finite whenever $M$ is not a torsion-module. Any ideal $J$ that realizes $M$ can be referred to as a \textit{partial trace ideal of $M$} (see \Cref{sec_traceideal}).

 We can restrict our study of $\h(\cdot)$ to ideals. More explicitly, for any ideal ${I}$ in a one-dimensional Noetherian local domain $R$ with quotient field $K$, we have 
\begin{align*}\label{defgeninv}
\ds \h({I}):=\min\{\lambda(R/J)~|~{I}\cong J\}.
\end{align*}

\begin{remark}\label{rem3.2}
	Two ideals $I$, $J$ are isomorphic means that there exists $\ds \alpha\in K$ such that $\ds I=\alpha J$ \cite[Lemma 2.4.1]{MR2266432}. (More generally, the same statement holds when $I$ and $J$ are $R$--submodules of $K$.)
\end{remark}

\begin{remark}\label{rem3.1}
	Note that for any finitely generated $R$-module $M$ with a non-zero rank, if $J$ realizes $M$ for some ideal $J$, then $\ds \h(J)=\h(M)$. In this case, also note that $J$ realizes itself.
\end{remark}

\begin{lemma}\label{prop.lem1}
	Let $\ds (R,\m,k)$ be a one-dimensional Noetherian local domain with integral closure $\overline{R}$. Further assume that $\ds \widehat{R}$ is reduced and let $x$ be a non-zero element in $R$. Then $\ds \lambda\({R}/{xR}\)=\lambda\({\overline{R}}/{x\overline{R}}\)$.	
\end{lemma}

\begin{proof}
	Since $\widehat{R}$ is reduced, $\ds \overline{R}$ is a finite $R$-module by  \cite[Corollary 4.6.2]{MR2266432}. Since $R$ and $\overline{R}$ have the same fraction field, $\Rank_R(\overline{R})=1$. Since $\ds \overline{R}$ is MCM over $R$, by \cite[Corollary 4.6.11]{bruns_herzog_1998}, we have $$\ds \lambda\({\overline{R}}/{x\overline{R}}\)=e(x;\overline{R})=e(x;R)\Rank_R\overline{R}=e(x;R)=\lambda\({R}/{xR}\).\hfill \qedhere$$
\end{proof}

\begin{theorem}\label{prop.propo1}
	Let $\ds (R,\m,k)$ be a one-dimensional Noetherian local domain with integral closure $\overline{R}$ and fraction field $K$. Further assume that $\ds \widehat{R}$ is reduced. For any non-zero ideal $J$ of $R$, consider the following statements:
	\begin{itemize}
		\item[$(a)$] $\ds \h(J)=\lambda\({R}/{J}\).$
		
		\item[$(b)$] $\ds R:_KJ\subseteq \overline{R}.$
	\end{itemize}  
	Then $(b)$ implies $(a)$. Further if $\ds \overline{R}$ is a $DVR$, then $(a)$ implies $(b)$.
	
\end{theorem}

\begin{proof}
	First, suppose $\ds R:_KJ\subseteq \overline{R}$. Suppose on the contrary that $\ds \h(J)=\lambda(R/I)<\lambda(R/J)$ for some ideal $I\cong J$. By \Cref{rem3.2}, there exists $a,b\in R$ such that $ I=\frac{a}{b}J$, or equivalently $\ds bI=aJ$. Now, $ \frac{a}{b}J=I\subseteq R$, and hence by assumption we have $ \frac{a}{b}\in \overline{R}$. We have 
	\begin{equation}\tag{\ref{prop.propo1}.1}\label{prop.eq1}
	\lambda\({R}/{J}\)+\lambda\({R}/{aR}\)=\lambda\({R}/{aJ}\)=\lambda\({R}/{bI}\)=\lambda\({R}/{I}\)+\lambda\({R}/{bR}\)
	\end{equation}
	Since $\lambda(R/I)<\lambda({R/J})$, we have $\lambda(R/aR)<\lambda(R/bR)$. By \Cref{prop.lem1}, we have $\lambda(\overline{R}/a\overline{R})<\lambda(\overline{R}/b\overline{R})$ but this is a contradiction as $\frac{a}{b}\in \overline{R}$. This finishes the first part of the proof. 
	
	For the remaining part, suppose $\ds \h(J)=\lambda(R/J)$ and let $\frac{a}{b}\in R:_KJ$. Suppose $\ds a\not \in b\overline {R}$. Since $\overline{R}$ is a $DVR$, we have $\ds b\overline{R}\subsetneq  a\overline{R}$. Set $I=\frac{a}{b}J$. By assumption, we have $\lambda(R/I)\geq \lambda(R/J)$. Hence, from \Cref{prop.eq1}, we get $\lambda(R/aR)\geq \lambda(R/bR)$ which by \Cref{prop.lem1}, implies that $\lambda(\overline{R}/a\overline{R})\geq \lambda(\overline{R}/b\overline{R})$. This is a contradiction to the fact that $\ds b\overline{R}\subsetneq  a\overline{R}$. This completes the proof.
\end{proof}
It should be interesting to explore whether $(a)$ implies $(b)$ in \Cref{prop.propo1} without the $DVR$ assumption. 

In light of the above theorem, we shall see that the conductor ideal naturally appears in the subsequent discussion. We quickly recall its definition first. 

For any reduced ring $R$ with total ring of fractions $K$, the conductor ideal of $R$ in $\overline{R}$ is  $\cC:=R:_K\overline{R}$ (the integral closure is taken in $K$).
It is the largest common ideal of $R$ and $\overline{R}$. It has a non-zero-divisor whenever $\overline{R}$ is finitely generated over $R$.

\begin{corollary}\label{prop.cor1}
	Suppose $\ds R$ is a one-dimensional Noetherian local domain with integral closure $\overline{R}$, fraction field $K$ and conductor $\cC$. Further assume that $\ds \widehat{R}$ is reduced. If $\cC\subseteq J$ for any ideal $J$, then $\ds \h(J)=\lambda(R/J)$. In particular, $\ds \h(\cC)=\lambda(R/\cC)$.
\end{corollary}

\begin{proof}
	It is well-known that $\ds \overline{R}=R:_K\cC$. We still present a proof here: 
	$ \overline{R}\subseteq R:_K\cC$ since $\cC\overline{R}\subseteq R$. For the other inclusion, note that for any $\alpha\in R:_K\cC$, we have $\alpha \in {R}:_K\cC\overline{R}$ and hence $\alpha\in \cC:_K\cC\subseteq \overline{R}$.
	
	The proof now follows from $(b)$ implies $(a)$ of \Cref{prop.propo1}. 
\end{proof}

\begin{corollary}\label{trace_cor_1}
	Let $\ds R$ be a one-dimensional Noetherian local domain with integral closure $\overline{R}$ and fraction field $K$. Further assume that $\ds \widehat{R}$ is reduced and $\overline{R}$ is a $DVR$. Let $M$ be a finitely generated  $R$-module with a non-zero rank and let  $J$ be an ideal that realizes $M$. Then for any ideal ${I}$ which contains $J$, we have $\ds \h({I})=\lambda\({R}/{{I}}\)$.
\end{corollary}

\begin{proof}
	By \Cref{rem3.1}, we have $\ds \h(M)=\h(J)$. By \Cref{prop.propo1}, we have $\ds R:_KJ\subseteq \overline{R}$. Since $\ds J\subseteq {I}$, we have $\ds R:_K{I}\subseteq R:_KJ\subseteq \overline{R}$. Another application of \Cref{prop.propo1} finishes the proof.
\end{proof}

\begin{remark}\label{convention_remark}
	Let $R$ be a domain with fraction field $K$. For any two $R$-submodules $I_1,I_2$ of $K$, we know by  \cite[Lemma 2.4.2]{MR2266432}, $\ds I_1:_KI_2\cong \Hom_R(I_2,I_1)$ where the isomorphism is as $R$-modules. 
	The proof of this isomorphism shows that the identification is as follows:
	any $\phi \in \Hom_R(I_2,I_1)$ {corresponds to multiplication by} ${\phi(x)}/{x}$ for some non-zero-divisor $x$ in $I_2$. 
	Henceforth in this paper, we are going to identify these two $R$-submodules of $K$ and use them interchangeably by abuse of notation. We use the notation $I^{-1}$ to denote $R:_K I$.
\end{remark}

\begin{lemma}\label{ann_lem}
	Let $R$ be a one-dimensional Cohen Macaulay local ring with canonical module $\omega$ and let $J$ be an ideal which contains a non-zero-divisor. Then 
	$$\ds \Ann_R\({\omega}/{J\omega}\)= x:_R(x:_RJ)$$ for every non-zero-divisor $x$ in $J$.

\end{lemma}

\begin{proof}
	First choose a non-zero-divisor $x\in J$. Let $\overline{\cdot}$ denote going modulo $xR$. Then $$\frac{\omega}{J\omega}\cong \frac{\omega_{R/xR}}{\overline{J}\omega_{R/xR}}\cong \frac{E}{\overline{J}E}$$ where $\ds E$ denotes the injective hull of $k$ as an $R/xR$ module. This follows from 
	\cite[Theorem 3.3.4, Theorem 3.3.5]{bruns_herzog_1998}.	
	Hence we may assume that $R$ is zero dimensional. It is now enough to show that the annihilator is $0:_R(0:_R J)$ (we  use $J$ again instead of writing $\overline{J}$, for convenience). (Note that this reduction follows because for any $x\in J$, $(x:_R J)/xR=0:_{R/xR}(J/xR)$, and we can apply it twice to get $(x:_R(x:_R J))/xR = 0:_{R/xR}(0:_{R/xR}(J/xR))$).
	\begin{align*}
	\Ann\({E}/{JE}\)&=\Ann\(\({E}/{JE}\)^\vee\) & \text{\cite[10.2.2(ii)]{brodmann2012local}}\\
	&=\Ann\(\Hom_R(E\tens R/J,E)\)\\
	&=\Ann\(\Hom_R(R/J,\Hom_R(E,E)\) & (\text{Hom-tensor adjointness})\\
	&= \Ann(\Hom_R(R/J,R))\\
	&=0:_R(0:_R J).  \qedhere
	\end{align*}
\end{proof}

\begin{theorem}\label{prop.thm1}
	Let $\ds R$ be a one-dimensional Noetherian local domain with integral closure $\overline{R}$, fraction field $K$ and conductor $\cC$. Further assume that $\ds \widehat{R}$ is reduced, $\overline{R}$ is a $DVR$ and the canonical module $\omega$ exists. Then for  any ideal $J$ of $R$, the following statements are equivalent.
	
	\begin{itemize}
		\item[$(a)$] $\ds \h(J)=\lambda\({R}/{J}\).$
		
		\item[$(b)$] $\cC\subseteq x:_R\(x:_R J\)$ for some $x \in J$.
		
		\item[$(c)$] $\ds \cC\omega\subseteq J\omega$.
	\end{itemize}
In particular, if $R$ is Gorenstein, then $\ds \cC\subseteq J$ for any $J$ which realizes itself. Thus if $R$ is Gorenstein, then for any finitely generated $R$-module $M$ of positive rank, we get $\h(M)\leq \h(\cC)=\lambda(R/\cC)$.
\end{theorem}

\begin{proof}
	Assume that statement $(a)$ holds. By \Cref{prop.propo1}, $\ds R:_K J\subseteq \overline{R}$. Let $\ds a\in x:_R J$. Now, $\ds aJ\subseteq (x)$ implies that $\frac{a}{x}J\subseteq R$ and hence, $\frac{a}{x}\in \overline{R}$. So, $\frac{a}{x}\cC\subseteq R$.

	Now assume $(b)$ holds. By \Cref{ann_lem}, we get $\ds \cC\subseteq \Ann_R({\omega}/{J\omega})$. Hence,  $\ds \cC\omega\subseteq J\omega$.
	
	Next let $(c)$ hold. Then we have $\ds \omega:_K J\omega\subseteq \omega:_K \cC\omega$ and this implies $\ds (\omega:_K \omega):_K J\subseteq (\omega:_K \omega):_K \cC$ by properties of colons. Since $R=\omega:_K\omega$ as submodules of $K$ (see \cite[Theorem 3.3.4(d)]{bruns_herzog_1998} and \Cref{convention_remark}), we have $R:_K J\subseteq R:_K \cC=\overline{R}$. This implies statement  $(a)$ by \Cref{prop.propo1}.
	
	When $R$ is Gorenstein, $\omega$ can be identified with $R$. Thus $\cC\subseteq J$ for any ideal $J$ which realizes itself, using  $(a)$ implies $(c)$ which we already established in the above discussion. Next notice that by \Cref{rem3.1}, we get that $\h(M)=\h(J)=\lambda(R/J)$ for any $J$ that realizes $M$. Again using $(a)$ implies $(c)$, we get that $\h(M)\leq \lambda(R/\cC)$. The proof is now complete by \Cref{prop.cor1}.
\end{proof}

Without the condition that $R$ is Gorenstein, the equivalent statements in \Cref{prop.thm1} do not imply that $\cC\subseteq J$. The following example illustrates this. The author would like to thank the reviewer for bringing this to attention.
\begin{example}\label{reviewer}
	Let $R=\frac{\CC[[X,Y,Z]]}{(Y^4-X^2Z,X^4-YZ,X^2Y^3-Z^2)}\cong \CC[[t^5,t^6,t^{14}]]$. Let $\m=(x,y,z)=(X,Y,Z)R$. Macaulay 2 computations show that $\cC=(z,x^3,x^2y,xy^2,y^3)$. Thus, $\cC\not\subseteq \m^2$. However, $\cC\subseteq x^2:_R(x^2:_R\m^2)$ as the latter is the ideal $(x^2,xy,y^2,z)$. 
	
\end{example}

\section{Relationship With The Trace Ideal}\label{sec_traceideal}

Note that if $\ds J$ realizes $M$ as in \Cref{definv}, we can discuss a relationship of $J$ to the trace ideal of $M$. We recall the definition of trace ideal.

\begin{definition}\label{deftrace}
	The trace ideal of an $R$-module $M$, denoted $\text{tr}_R(M)$, is the ideal $\ds \sum \alpha(M)$ as $\alpha$ ranges over $\ds M^{*}:= \Hom_R(M,R)$. The \textit{trace map} is defined as follows. \begin{align*}\text{Tr}_M:~  &M\tens_RM^*\to R\\
	&\text{Tr}_M(m\tens \alpha)=\alpha(m)
	\end{align*}
	Note that the image of the trace map is $\text{tr}_R(M)$. 
\end{definition}	

Notice that any ideal $J$ that realizes $M$ is the image of one such map $\alpha$ as in \Cref{deftrace}, which justifies thinking of $J$ as a partial trace ideal of $M$.

\begin{remark}\label{vasconcelos} An ideal $I$ is called a trace ideal if $I=\text{tr}_R(M)$ for some $R$-module $M$. For a summary of some of the standard properties of trace ideals, the reader can refer to  \cite[Proposition 2.8]{MR3646286} and \cite[Proposition 2.4]{kobayashi2019rings}. We record the following facts which will be useful for us.
	\begin{itemize}
		\item[a)] For any ideal $I$, $I\subseteq \text{tr}_R(I)$. Also, if a module $M$ surjects to $I$, it is evident from the definition that $I\subseteq \text{tr}_R(M)$. Hence, $I\subseteq \text{tr}_R(I)\subseteq \text{tr}_R(M)$. 
		\item[b)]\label{trace_compute_rem1} Suppose $\Phi$ is a presentation matrix for a module $M$ and $\Psi$ is a matrix whose columns generate the kernel of $\Phi^t$. Then $\ds \text{tr}_R(M) =I_1(\Psi)$ where $I_1(\Psi)$ is the ideal generated by the entries of $\Psi$; see \cite[Remark 3.3]{MR1055780}  and \cite[Proposition 3.1]{MR4013970}.
		
		\item[c)] For any regular fractional ideal $I$ (i.e., an $R$-submodule of the ring of quotients of $R$, which contains a non-zero-divisor of $R$), we have $\text{tr}_R(I)=I(R:_KI)=II^{-1}$ (see \cite[Proposition 2.4]{kobayashi2019rings}).
	\end{itemize}
\end{remark}

\begin{example}
	Let $R=k[x,y,z]$. Let $\ds I=(xy,yz)$. A presentation of $I$ is given by 
	$$R\xrightarrow{
		\begin{bmatrix}
		-z\\
		x
		\end{bmatrix}}
	R^2\xrightarrow{
		\begin{bmatrix}
		xy & yz
		\end{bmatrix}}
	I\to 0$$
	
	\noindent Following the notation of \Cref{trace_compute_rem1}(b), $\ds \Phi^t=[-z~~x]$; so $\ds \Psi=[x ~z]^t$. Hence, $\text{tr}_R(I)=(x,z)$. In fact, this is an example where $I\neq \text{tr}_R(I)$ which in turn implies that $I$ is not a trace ideal \cite[Proposition 2.8]{MR3646286}; however, we  have that $\ds I\cong \text{tr}_R(I)$. 
\end{example}

\begin{proposition}\label{trace_prop}
	Let $R$ be a Noetherian one-dimensional local domain and let $M$ be a finitely generated $R$-module with non-zero rank. Then
	$\ds \h(M)\geq \lambda\({R}/{\text{tr}_R(M)}\)\geq \h(\text{tr}_R(M))$. If in addition  $\widehat{R}$ is reduced, then $\ds \h(\text{tr}_R(M))= \lambda(R/\text{tr}_R(M)).$

\end{proposition}

\begin{proof}
	Choose an ideal $J$ which realizes $M$. The first part of the proof is now complete by \Cref{vasconcelos}(a) and \Cref{definv}. The last equality follows from \Cref{prop.propo1} since for any trace ideal $I$, $R:_K I=I:_K I\subseteq \overline{R}$ \cite[Proposition 2.8(vi)]{MR3646286}.
\end{proof}

We can in fact completely describe the trace of any such partial trace ideal for any finitely generated module of rank one. 

\begin{proposition}\label{trace_rem1}
	Let $R$ be a Noetherian one-dimensional local domain and let $M$ be a finitely generated $R$-module of rank one. Then for any ideal $J$ such that $M$ surjects to $J$, we have $$\ds \text{tr}_R (J)=\text{tr}_R (M).$$
	
\end{proposition}

\begin{proof} For any ideal $J$ such that $M$ surjects to $J$, we have the following exact sequence:
	\begin{equation}\tag{\ref{trace_rem1}.1}\label{trace_rem1_eq0}
		0\to \tau(M)\to M\xrightarrow{f} J\to 0
	\end{equation} \noindent where $\tau(M)$ is the torsion submodule of $M$. This is because $J$ is torsion-free and $\Rank_R(J)=1=\Rank_R(M)$. 
	Applying $\ds \Hom_R(\cdot,R)$ to \Cref{trace_rem1_eq0}, we have the isomorphism $\ds f^*:J^*\to M^*$. Thus every map $M \to R$
	factors through a map $J\to R$ and hence, $\ds \text{tr}_R(M) \subseteq \text{tr}_R(J)$. The other containment is clear.
	\end{proof}

\section{Application to $\Omega_{R/k}$: Berger's Conjecture}\label{invariantapp}
	
In this section, we are going to focus on $\hun$. Throughout this section, mainly we shall need $\ds (R,\m ,k)$ to be a Noetherian complete local one-dimensional domain such that $\ds R=S/I$ where $S=k[[X_1,\dots,X_n]]$, $k$ is perfect and $n\geq 2$. Let $\n$, $\cC$ and $\omega$ be the maximal ideal of $S$,  the conductor ideal of $R$ and the canonical module of $R$ respectively. Finally, $\ds \Omega_{ R/k}$ denotes the universally finite module of differentials (see \Cref{defOmega}) and $\tau(\Omega_{R/k})$ denotes its torsion submodule.

Note that under these assumptions on $R$, the hypothesis of \Cref{prop.propo1} is satisfied and we can apply the results of the previous sections. We summarize them in the next proposition.  

\begin{proposition}\label{allboundssummary}
Let	$\ds (R,\m ,k)$ be a Noetherian complete local one-dimensional domain such that $\ds R=S/I$ where $S=k[[X_1,\dots,X_n]], \ k$ is perfect  and $\ds I\subseteq (X_1,\dots,X_n)^{2}$.   Let $J$ be an ideal that realizes $\ds \Omega_{ R/k}$. Then the following hold.
\begin{itemize}
	\item[$(a)$] $\ds \cC\subseteq x:_R(x:_R J)$ for some $x\in J$.
	
	\item[$(b)$] $\ds \cC\omega\subseteq J\omega.$
	
	\item[$(c)$] If $R$ is Gorenstein, then $\ds \cC\subseteq J$.

	\item[$(d)$] If $R$ is Gorenstein, then $\ds \hun\leq \h(\cC)=\lambda(R/\cC)$.

	\item[$(e)$] $\ds \h(\Omega_{ R/k})\geq \h(\text{tr}_R(\Omega_{ R/k}))=\lambda(R/\text{tr}_R(\Omega_{ R/k}))=\lambda(R/\text{tr}_R(J))=\h(\text{tr}_R(J))$.

\end{itemize}
\end{proposition}	

\begin{proof}
	The statements $(a),(b),(c)$ and $(d)$ follow immediately from \Cref{prop.thm1}. The statement $(e)$ follows from 
	\Cref{trace_prop} and \Cref{trace_rem1}.
\end{proof}

We are only going to focus on the non-regular case, i.e., $\mu(\m)\geq 2$. The following remark shows that in order to study $\tau(\Omega_{R/k})$, we may only consider the cases when $\ds \hun\geq 1$.

\begin{remark}\label{h=0}
		Suppose $\ds \hun=0$. Then we have an exact sequence
		$$0\to \tau(\Omega_{R/k}) \to \Omega_{ R/k}\to R\to 0,$$ thus $R$ splits off $\ds \Omega_{R/k}$.  If $R$ is non-regular, then $\tau(\Omega_{R/k})\neq 0$ as otherwise, $\m$ should be principally generated, a contradiction. In fact, whenever $R$ is complete and $\chr(k)=0$, any such surjection from $\Omega_{ R/k}$ to $R$ implies that $R$ is regular. This can be found, for instance, in the Nordic Summer School Lecture notes of B. Teissier \cite[discussion on pages 586-587]{teissier1976hunting}. This is essentially a generalization of the proof with $\ds k=\mathbb{C}$ which can be found in the work of H. Hironaka \cite[Lemma 5]{hironaka1983nash}. Both authors attribute the proofs to O. Zariski.

\end{remark}

We make the following remark which will be useful in the discussions that follow.
\begin{remark}\label{maximalidealgen}
	Suppose $\ds (S,\n,k)$ is a regular local ring of embedding dimension $n$ and $\ds R=S/I$ for some ideal $I$ in $S$ where $\ds I\subseteq \n^{s+1}$ for some $s\geq 1$. Letting $\m$ denote the maximal ideal of $R$, we have $\mu(\n^i)=\mu(\m^i)={n+i-1\choose i}$ for $1\leq i\leq s$ due to the condition imposed on $I$.
\end{remark}

We start by proving three results which will be crucial to the proof of the main theorem.

	\begin{lemma}\label{lem2}
	Let	$\ds (R,\m ,k)$ be a Noetherian complete local one-dimensional domain such that $\ds R=S/I$ where $S=k[[X_1,\dots,X_n]]$, $k$ is perfect and $n\geq 2$. Let $\n$ be the maximal ideal of $S$. Suppose $J$ realizes $\ds \Omega_{R/k}$. If $\ds I\subseteq \n ^{s+1}$ for some $s\geq 1$ and $\tau(\Omega_{R/k})=0$, then $\ds J\subseteq \m^s$.
	\end{lemma}
	
	\begin{proof}
	Since $J$ is torsion-free and since $\Rank_R(J)=1=\Rank_R(\Omega_{ R/k})$,	we have the following exact sequence:
		\begin{equation}\tag{\ref{lem2}.1}\label{trace_rem1_eq1}
			0\to \tau(\Omega_{R/k})\to \Omega_{ R/k}\to J\to 0.
		\end{equation}    
 Since $\tau(\Omega_{ R/k})=0$, $\ds \Omega_{R/k}\cong J$. Using notations as in \Cref{defOmega},  any relation between the generators of $J$ must come from the matrix $A$ presenting $\Omega_{ R/k}$. Suppose if possible, that there exists $y_1,y_2,\dots,y_n$ generators of $J$ such that $y_1\not \in \m^s$. There exists the Koszul relation, $\ds -y_2y_1+y_1y_2=0$. Hence, the column vector $\ds [-y_2\quad y_1\quad 0 \quad\dots\quad 0]^t$ of $R^n$ must be written as an $R$-linear combination of the columns of $A$. Since $I\subseteq \n^{s+1}$, the entries of $A$ are in $\n^s$ but this is a contradiction to the choice of $y_1$.
	\end{proof}

	The next proposition is a generalization of \cite[Theorem 2.3.2(c)]{bruns_herzog_1998}. Before stating it, we recall that the \textit{associated graded ring} corresponding to the maximal ideal $\n$ in a local ring $(S,\n,k)$ is defined to be \[\operatorname{gr}_\n(S):=\oplus_{i\geq 0}\frac{\n^i}{\n^{i+1}}.\] For any element $s\in S$, we choose the greatest integer $\ell$ such that $s\in \n^\ell$, and define the leading term of $s$ to be $s^*:=s~\text{modulo~}\n^{\ell+1}\in \frac{\n^\ell}{\n^{\ell+1}}
	$. Hence, $s^*\in \operatorname{gr}_\n(S)$. If $S$ is a regular local ring with embedding dimension $n$, then $\operatorname{gr}_\n(S)$ is isomorphic to a polynomial ring over $k$ with $n$ variables. This follows from \cite[Theorem 1.1.8]{bruns_herzog_1998}. For further details, see \cite[5.1]{eisenbud2013commutative}.

\begin{proposition}\label{lem1}
	Let $\ds (S,\n,k)$ be a regular local ring of embedding dimension $n$ and let $\ds R=S/I$ for an ideal $I$ in $S$. Let $\m$ be the maximal ideal of $R$. Further assume that $\ds I\subseteq \n^{s+1}$ for some $s\geq 1$. Then $$\ds \text{dim}_k\(\Tor^R_1(\m^s,k)\)=s{n+s-1\choose s+1}+\mu(I).$$
\end{proposition}

\begin{proof} Choose $\ds x_1,\dots,x_n$ to be a minimal system of generators of $\ds \m$. Let $\ds y_1,y_2,\dots,y_n$ be a regular system of parameters in $S$ such that $\ds \overline{y_i}=x_i$, where $\ds ~~ \bar{\cdot}~~$ denotes going modulo $I$. 
	
	Using the short exact sequence $0\rightarrow \m^s\rightarrow R\rightarrow R/\m^s\rightarrow 0$, we get $\Tor^R_1(\m^s,k)\cong\Tor^R_2(R/\m^s,k)$. 
	In order to get the desired conclusion, we shall make use of the minimal resolutions of $R/\m^s$ and $S/\n^s$. Recall that by \Cref{maximalidealgen}, we have $\mu(\n^s)=\mu(\m^s)={n+s-1\choose s}$.

	Let $F$ denote the free module $S^{n+s-1},G$ denote the free module $S^s$ and $\ds L:F\to G$ be given by the $s\times (n+s-1)$ matrix
	\begin{equation*}
	L=\begin{bmatrix}
	y_1  & y_2 &...& y_n & 0 &... & 0\\
	0 & y_1 & y_2 & ... & y_n & ... &0\\
	\vdots & \vdots &\vdots &\vdots &\vdots &\vdots &\vdots\\
	0 & ... & 0 & y_1 & y_2 & ... &y_n
	\end{bmatrix}.
	\end{equation*}
	By the proof following Remark 2.13 in \cite{bruns2006determinantal}, the ideal of $s\times s$ minors of $L$ is $\n^s$. Since $S$ is a regular local ring and $\operatorname{height}(\n^s)=n=\Rank_S(F)-\Rank_S(G)+1$, we get that the Eagon-Northcott complex 
	$$\ds 
	{\mathbb{EN}}(L)_\bullet: \cdots \to  G^*\tens_S\wedge^{s+1}F \xrightarrow{d_2} \wedge^sF \xrightarrow{d_1} \wedge^sG\cong S \to 0
	$$
	resolves $S/\n^s$ \cite[Theorem A2.60]{MR2103875} (for further details on the $\operatorname{height}$ of an ideal, we refer the reader to \cite[Appendix]{bruns_herzog_1998}). Fix $p={n+s-1\choose s}=\Rank_S(\wedge^sF)$ and $q=s{n+s-1\choose s+1}=\Rank_S(G^*\tens\wedge^{s+1}F)$. Using the identifications $G^*\tens_S\wedge^{s+1}F\cong S^q, \wedge^sF\cong S^{p}$ and $\wedge^sG\cong S$, we can rewrite the above complex as 
	$$\ds 
	{\mathbb{EN}}(L)_\bullet: \cdots \to  P_2 \xrightarrow{d_2} P_1 \xrightarrow{d_1} P_0 \to 0
	$$
	where $P_2=S^q,P_1=S^p,P_0=S$. Moreover, with respect to the standard bases of $P_1$ and $P_0$, the map $d_1$ is given by the matrix $[Y_1 \cdots Y_p]$ where $Y_i\in \n$ are the $s\times s$ minors of $L$. Notice that these minimally generate $\n^s$. We are going to produce elements $\{w_j\}_{1\leq j\leq q}$ of $P_2$ and $\{v_i\}_{1\leq i\leq \mu(I)}$ of $P_1$  which will help in establishing the required dimension.

	Let $\mu(I)=m$ and $I$ be minimally generated by $c_1,\dots,c_m$. 
	Since $ I\subseteq \n^{s+1}\subseteq \n^s$, we can write $ c_i=\sum\limits_{k=1}^{p}b_{ki}Y_k$ with $b_{ki}\in \n$. We write this in row format and set $C=[c_1\cdots c_m]$.  We get that \begin{equation}\tag{\ref{lem1}.1}\label{keyeq}
	C=\begin{bmatrix}
	Y_1
	\cdots
	Y_p
	\end{bmatrix}B=d_1B\end{equation} where $\ds B:=(b_{ki})_{1\leq k\leq p\atop1\leq i\leq m}$ is a $p\times m$ matrix with entries in $\n$. Let $\{e_i\}$ denote the standard basis of $S^m$. We have $Ce_i=c_i$ and thus using  \Cref{keyeq}, we get the commutative diagram
	\[
	\begin{tikzcd}[
	ar symbol/.style = {draw=none,"\textstyle#1" description,sloped}, isomorphic/.style = {ar symbol={\simeq}}]
	P_1\arrow[r,"d_1"] & P_0\\
	S^m\arrow[u,"B"]\arrow[r,"C",swap]& S\arrow[u,isomorphic]
	\end{tikzcd}
	\]
	Let $v_i=Be_i$. Note that these are the columns of $B$. By commutativity of the diagram,  we also have that $c_j=Ce_j=d_1(Be_j)=d_1(v_j)$.
	
	Tensoring the Eagon-Northcott complex with $R$, we obtain the exact sequence
	$$ 0\to Z_1 \to R^{p}\xrightarrow{\overline{d_1}} R\to R/\m^s\to 0$$ where $\overline{d_1}=d_1\otimes_S R=[\overline{Y_1}\cdots \overline{Y_p}]$ and $Z_1=\ker \overline{d_1}$. Notice that the second and third terms in the above exact sequence are the starting terms of the minimal resolution of $R/\m^s$. Moreover, $\overline{d_1}$ is the corresponding presentation matrix. Thus, we get that $\dim_k\Tor^R_1(\m^s,k)=\dim_{k}\Tor^R_2(R/\m^s,k)= \mu(Z_1)$.
	
	Let $b\in P_1$ be such that $\overline{b}\in Z_1$. Then $d_1(b)\in I$ and hence, $d_1(b)=\sum_{i=1}^ms_ic_i$, $s_i\in S$. Since $d_1(v_i)=c_i$, we get \[d_1(b)=\sum_{i=1}^ms_ic_i=\sum_{i=1}^ms_id_1(v_i)=d_1(\sum_{i=1}^ms_iv_i)\] and thus $b-\sum_{i=1}^ms_iv_i\in \ker(d_1)=\text{Im}(d_2)$, where $P_2 \xrightarrow{d_2} P_1$ is the second differential of ${\mathbb{EN}}(L)_\bullet$.  Denoting the standard basis of $\ds P_2$ by $\{w_i\},1\leq i\leq q $, we can write $b-\sum_{i=1}^ms_iv_i= d_2(\sum_{i=1}^q s_i'w_i)$, $s_i'\in S$. So $$\overline{b}=\sum_{i=1}^q\overline{s_i'}\overline{d_2(w_i)} + \sum_{i=1}^m\overline{s_i}\overline{v_i}.$$ Since $\ker (d_1)=\text{Im}(d_2)$, clearly we have $\overline{d_2(w_i)} \in Z_1$ for $1\leq i\leq q$. Since  $c_j=d_1(v_j)$, $\overline{d_1}(\overline{v_j})=0$ and hence $\overline{v_j}$ is in $Z_1$ for $1\leq j\leq m$. Thus we obtain that $$\mu(Z_1)\leq q+m=s{n+s-1\choose s+1}+\mu(I).$$
	
	We now show that $\{\overline{d_2(w_i)},\overline{v_j}\}$ for $1\leq i\leq q,1\leq j\leq m$ form a minimal generating set of $Z_1$. Suppose $ \sum\limits_{i=1}^{q}\overline{\alpha_id_2(w_i)}+\sum\limits_{j=1}^m\overline{\beta_jv_j}=0$ where ${\alpha_i},{\beta_j}\in S$. We are done if we show that all the $\overline{\alpha_i}$ and $\overline{\beta_j}$ are in $\m$, i.e., $\alpha_i,\beta_j$ are in $\n$. The relation under consideration implies that \begin{equation}\tag{\ref{lem1}.2}\label{eq0}\ds \sum\limits_{i=1}^{q}{\alpha_id_2(w_i)}+\sum\limits_{j=1}^m{\beta_jv_j}\in IP_1.
	\end{equation} Applying $d_1$ and again using $d_1(v_i)=c_i$, we get
	\begin{align*}
	\sum\limits_{i=1}^{q}{\alpha_id_1(d_2(w_i))}+\sum\limits_{j=1}^m{\beta_jd_1(v_j)}=\sum_{j=1}^m\beta_jc_j\in Id_1(P_1).
	\end{align*}  
	
	\noindent As $d_1(P_1)\subseteq \n^s$, we obtain that
	$ \sum_{j=1}^{m}\beta_jc_j\in I\n^s\subseteq \n I$. Since $c_j$'s minimally generate $I$, this implies that $\beta_j\in \n$. Since the entries of $B$ are in $\n$,   $v_j=Be_j\in \n P_1$. Therefore $ \sum_{j=1}^m\beta_jv_j\in \n^2P_1$. Now from (\ref{eq0}), we obtain that \begin{align}\tag{\ref{lem1}.3}\label{oops} \sum_{i=1}^{q}{\alpha_id_2(w_i)}\in IP_1+\n^2P_1= (I+\n^{2})P_1\subseteq \n^{2}P_1\end{align} since $I\subseteq \n^{s+1}$ for $s\geq 1$.

	We are trying to show that $\alpha_i\in \n$. We may assume without loss of generality that $\alpha_{1},\dots,\alpha_{\ell}\not\in \n$ and  the remaining $\alpha_j$'s are in $\n$. Recall that the entries of $d_2$ are linear forms in the entries of $L$ (see discussions in \cite[A2.6.1]{eisenbud2013commutative} and \cite[Example A2.69]{MR2103875}). From (\ref{oops}), we conclude that 
	\begin{equation}\tag{\ref{lem1}.4}\label{specialized}\sum_{i=1}^\ell\alpha_{i}d_2(w_{i}) \in \n^2P_1.\end{equation} 
	
	Now observe that $\operatorname{gr}_\n(S)$ is isomorphic to a polynomial ring in $n$ variables over $k$. Consider the matrix \begin{equation*}
	L'=\begin{bmatrix}
	y_1^*  & y_2^* &...& y_n^* & 0 &... & 0\\
	0 & y_1^* & y_2^* & ... & y_n^* & ... &0\\
	\vdots & \vdots &\vdots &\vdots &\vdots &\vdots &\vdots\\
	0 & ... & 0 & y_1^* & y_2^* & ... &y_n^*
	\end{bmatrix}
	\end{equation*} where the $y_i^*$ are the leading forms of $y_i$ (recall that the $y_i$'s minimally generate $\n$). Since the $y_i^*$ form a regular sequence and generate the homogeneous maximal ideal  of $\operatorname{gr}_\n(S)$, we get that the Eagon-Northcott complex over $\operatorname{gr}_\n(S)$ corresponding to $L'$, namely
	$$\ds 
	{\mathbb{EN}}(L')_\bullet: \cdots \to   P_2' \xrightarrow{d_2'}  P_1' \xrightarrow{d_1'}  P_0' \to 0,
	$$
	is also exact (here the modules $P_i'\cong \operatorname{gr}_\n P_i$ are the corresponding free modules over $\operatorname{gr}_\n(S)$).  
	Let $\{w_i'\}$  denote the canonical basis of $P_2'$.
	
	Again, for $i\geq 2$, the entries of $d_i'$ are linear forms in the entries of $L'$. Thus upon comparing degrees, from (\ref{specialized}) we get that $\sum_{i=1}^\ell\alpha_{i}^*d_2'(w_{i}')=0$, where the $\alpha_{i}^*\in k\setminus\{0\}$ are the leading forms of $\alpha_{i}$. In other words, $$\sum_{i=1}^\ell\alpha_{i}^*w_{i}'\in \ker(d'_2)=\operatorname{Im}(d_3')\subseteq \mathfrak{N} P_2'$$ where $\mathfrak{N}$ denotes the homogeneous maximal ideal of $\operatorname{gr}_\n(S)$.  This contradicts the fact that the $\{w_i'\}$ are a basis of $P_2'$.
	
	Thus, $\alpha_{i}^*=0$ for all $1\leq i\leq \ell$. This in turn implies that $\alpha_{i}\in \n$ for all $i$, thereby finishing the proof of \Cref{lem1}.
\end{proof}

\begin{lemma}\label{lengthlemma}
	Let $(R,\m,k)$ be a Noetherian local ring and $M$ be an $R$-module of finite length. Then $$\ds \text{dim}_k\Tor^R_1 (M, k) \leq \lambda(M )(\dim_k\Tor^R_1(k,k)-1)+\mu(M)\leq \lambda(M)\dim_k\Tor^R_1(k,k).$$  
\end{lemma}
\begin{proof} First note that the last inequality follows as $\mu(M)\leq \lambda(M)$. Set $\lambda = \lambda(M)$ and start with a composition series  $0=M_0\subsetneq M_1\subsetneq \dots \subsetneq M_{\lambda-1}\subsetneq M_{\lambda}=M$. Tensor $k$ with the short exact sequence $0\to M_{\lambda-1}\to M\to M_{\lambda}/M_{\lambda-1}\cong k \to 0$ to get a long exact sequence of $\Tor$s which gives us $$\ds \dim_{k}\Tor^R_1(M,k)\leq \dim_{k}\Tor^R_1(M_{\lambda-1},k)+\dim_k\Tor^R_1(k,k)+\mu(M)-\mu(M_{\lambda-1)})-1.$$ 
We	keep repeating this procedure for $M_{\lambda-1},\dots, M_{1}$ successively to obtain that $$\dim_k\Tor^R_1(M,k)\leq \lambda\dim_k\Tor^R_1(k,k)+\mu(M)-\lambda. \qedhere$$
\end{proof}

The following theorem is one of the main results of this paper. It identifies an upper bound of $\hun$ which will give some cases of Berger's conjecture. 
	
	\begin{theorem}\label{mainthm}
		Let	$\ds (R,\m ,k)$ be a Noetherian complete local one-dimensional domain such that $\ds R=S/I$ where $S=k[[X_1,\dots,X_n]]$, $k$ is perfect and $n\geq 2$. Let $\n$ be the maximal ideal of $S$ and assume that $\ds I\subseteq \n^{s+1}$ for $s\geq 1$. Let $J$ be any ideal that realizes $\Omega_{ R/k}$. If  $$\ds \h(\Omega_{R/k})<
		 {{n+s\choose s}\frac{s}{s+1}},$$ then
		$\ds \tau(\Omega_{R/k})\neq 0$.
	\end{theorem}
	
	\begin{proof} 
		 Suppose on the contrary that $\tau(\Omega_{ R/k})=0$. Choose an ideal $\ds J$ which realizes $\Omega_{ R/k}$; 
		by \Cref{lem2}, $ J\subseteq \m^s$. Using additivity of length, we have $ \lambda\({\m^s}/{J}\)=\hun -1-\sum_{i=1}^{s-1}\mu(\m^i).$ By \Cref{maximalidealgen}, we have $ \mu(\m^i)={n+i-1\choose i}$ for $1\leq i\leq s$ and hence using identities on binomial coefficients, we obtain 
		\begin{equation}\tag{\ref{mainthm}.1}\label{eq1}
	 \lambda\({\m^s}/{J}\)=\hun-\sum_{i=0}^{s-1}{n-1+i\choose i}=\hun-{n+s-1\choose s-1}.
		\end{equation}
		 Using \Cref{lengthlemma}, we get $\ds\dim_k\Tor^R_1\(\frac{\m^s}{J},k\)\leq \lambda\({\m^s}/{J}\)\dim_{k}\Tor^R_1(k,k)$.  Thus  from \Cref{eq1} we get,
		\begin{equation}\tag{\ref{mainthm}.2}\label{eq2}
		\dim_k\Tor^R_1\(\frac{\m^s}{J},k\)\leq n\(\hun-{n+s-1\choose s-1}\).
		\end{equation} 
		Tensoring the exact sequence $\ds 0\to J\to \m^s\to {\m^s}/{J}\to 0$  with $\ds k$, we get 
		$$\Tor^R_1(J,k)\to \Tor^R_1(\m^s,k)\to \Tor^R_1\(\frac{\m^s}{J},k\)\to J\tens k\to \m^s\tens k \to \frac{\m^s}{J}\tens k\to 0 $$ as part of a long exact sequence. Note that $D:=\mu(J)-\mu(\m^s)+\mu(\m^s/J)\geq 0$. Hence, we have
		\begin{equation}\tag{\ref{mainthm}.3}\label{eq3}
		\text{dim}_k\Tor^R_1\(\frac{\m^s}{J},k\)\geq -\text{dim}_k\Tor^R_1(J,k)+\text{dim}_k\Tor^R_1(\m^s,k)  +D.
		\end{equation}
		
		From the exact sequence $\ds 0\to \tau(\Omega_{R/k})\to \Omega_{R/k}\to J\to 0$, we get that $\ds \Omega_{R/k}\cong J$ if $\tau(\Omega_{ R/k})=0$. Hence, $\ds \text{dim}_k\Tor^R_1(J,k)\leq \mu(I)$ by \Cref{rankomega}$(c)$. Combining these observations together with \Cref{lem1}, \Cref{eq2} and \Cref{eq3}, we get 
		\begin{equation*}
		n\(\hun-{n+s-1\choose s-1}\)\geq -\text{dim}_k\Tor^R_1(J,k)+\text{dim}_k\Tor^R_1(\m^s,k)\geq  s{n+s-1\choose s+1}.
		\end{equation*}
		\text{Hence,~}
		\begin{align*}
		 \hun &\geq \frac{s(n+s-1)!}{n(s+1)!(n-2)!}+\frac{(n+s-1)!}{(s-1)!n!}\\
		&=\frac{(n+s-1)!}{(s-1)!n(n-2)!}\(\frac{1}{s+1}+\frac{1}{n-1}\)\\
		&=\frac{(n+s-1)!}{(s-1)!n!}\(\frac{n+s}{s+1}\)\\
		&=\frac{(n+s)!}{s!n!}\(\frac{s}{s+1}\)\\
		&={n+s\choose s}\(\frac{s}{s+1}\)
		\end{align*}
		This contradicts the hypothesis on $\hun$. Hence, $\tau(\Omega_{ R/k})\neq 0$.
	\end{proof}

\begin{remark}\label{mainrmk}
	We make a few observations in connection with the above proof. 
	  \begin{itemize}
	  	\item[a)] The exact sequence $\tau(\Omega_{R/k})\tens_Rk\to \Omega_{ R/k}\tens_Rk\to J\tens_Rk\to 0$ implies that $\ds \mu(\tau(\Omega_{R/k}))\geq \mu(\Omega_{ R/k})-\mu(J)=n-\mu(J)$. So, if  $\ds \Omega_{ R/k}$ surjects to an ideal $J$ which is generated by less than $n$ elements, then $\tau(\Omega_{R/k})\neq 0$.

		\item[b)] 
		 We can repeat the proof of \Cref{mainthm} up to \Cref{eq3} and use the observation that $D\geq n-{n+s-1\choose s}+\mu(\m^s/J)\geq 0$ (under the assumption that $\tau(\Omega_{ R/k})=0$). This helps us to obtain an alternate lower bound for $\hun$:{\small 
		 	\begin{align*}
		 	\hun & \geq {n+s\choose s}\frac{s}{s+1}+\frac{D}{n}
		\ge {n+s\choose s}\frac{s}{s+1}-\frac{{n+s-1\choose s}}{n}+1+\frac{1}{n}\mu(\m^s/J)\\
		&= {n+s-1\choose s-1}\frac{s^2+s(n-1)-1}{s(s+1)}+1+\frac{1}{n}\mu\({\m^s}/{J}\)
\end{align*}}Thus, $\tau(\Omega_{ R/k})\neq 0$ if $\hun<{n+s-1\choose s-1}\frac{s^2+s(n-1)-1}{s(s+1)}+1+\frac{1}{n}\mu\({\m^s}/{J}\)$.

\end{itemize}
\end{remark}

From existing results due to J. Herzog, we now immediately obtain the following corollary.

\begin{corollary}\label{maincor}
	Let	$\ds (R,\m ,k)$ be a Noetherian complete local one-dimensional domain such that $\ds R=S/I$ where $S=k[[X_1,\dots,X_n]]$, $k$ is perfect and $n\geq 2$. Assume that $\ds I\subseteq (X_1,\dots,X_n)^2$. Then the following statements hold.
	\begin{itemize} 
	
	\item[$(a)$] $\tau(\Omega_{R/k})\neq 0$ if $\ds \h(\Omega_{ R/k})=1,2$.
	
	\item[$(b)$] If $R$ is Gorenstein, then $\tau(\Omega_{R/k})\neq 0$ if $\h(\Omega_{R/k})=1,2,3$. 
\end{itemize}
\end{corollary}

\begin{proof}
	We can always assume that $n\geq 4$ due to \cite[Satz 3.3]{MR512469}.
Statement $(a)$ now follows immediately from  \Cref{mainthm} by taking $\ds s=1$.
	
	For $(b)$, note that by  \cite[Satz 3.2]{MR512469}, we can assume $n\geq 5$. From (a), we only need to take care of $\hun=3$. Take $s=1,n=5$ in \Cref{mainrmk}$(b)$ and note that $\ds J\neq \m$ since $\hun=3$.  
\end{proof}

\begin{remark}\label{corScheja} In \cite{scheja1970differentialmoduln}, the term \textit{quasi-homogeneous} was coined for any such 
$R$ where a surjective
R-module homomorphism $\ds \Omega_{ R/k}\to  \m$ exists. In this case, $\hun=1$ (provided $R$ is non-regular) and \Cref{maincor} can be applied. In particular, this recovers the standard graded case over a field of characteristic $0$ which was originally proved by G. Scheja \cite[Satz 9.8]{scheja1970differentialmoduln}. For further discussions on the quasi-homogeneous case, we refer the reader to \cite[2.3]{berger1994report}.  
 \end{remark}

\section{Algorithm for Explicit Computation}\label{explicitsec}
We shall now discuss an explicit way of computing $\hun$ which can help in providing examples of classes of rings where $\hun$ is relatively small compared to the embedding dimension and hence \Cref{mainthm} can be applied.

\begin{remark} \label{hcomprem1}
	From now on, we assume that $k$ is algebraically closed of characteristic $0$. Since $R$ is a one-dimensional complete local domain which is a $k$-algebra, the integral closure can be written as $\overline{R} = k[[t]]$ for some choice of uniformizing parameter $t$ \cite[Theorem 4.3.4]{MR2266432}. We use this description of $\overline{R}$ henceforth. Identifying $R$ inside $\overline{R}$, we have that each generator $x_i=x_i(t)$ of $\m$ is a power series in $t$. Further, we shall use the \textit{valuation function} $v$ on $\overline{R}$ that takes a power series in $t$ as input and outputs the lowest power of $t$ in the series. This is a discrete valuation on $k((t))$. We need the following observations.

	\begin{enumerate}  
		\item Let $\ds \mathcal{D}=Rx_1'(t)+\dots +Rx_n'(t)$ be the `fractional ideal of derivatives' obtained by differentiating each of the generators of $\m$ with respect to $t$ and then taking the $R$-span. Thus, $\mathcal{D}$ is a fractional ideal, i.e., an $R$-submodule of $k((t))$. Now $\ds \Omega_{R/k}$ surjects to $\mathcal{D}$. This follows from the chain rule of derivations if we ignore the $dt$ (see, for example, the discussion at the beginning of \cite[Section 3]{MR629286}). Multiplying by a suitable power series in $t$, we can get an ideal in $R$ which is an isomorphic copy of $\mathcal{D}$. Hence, we have constructed an ideal of $R$ to which $\Omega_{ R/k}$ surjects. 
		
		\item In the notations of \cite{berger1994report}, we have $S=\overline{R}$ and $DS=dt$. Since $S$ is regular, $SDS=\Omega_{S/k}$ is free of rank $1$ and hence can be identified with $S$. Note that this identification is obtained by simply replacing the basis vector $dt$ with $1$. Via the same isomorphism, $\mathcal{D}$ is now identified with $RDR$. Hence, $\lambda(\overline{R}/\mathcal{D})=\lambda(S/\mathcal{D})=\lambda({SDS}/RDR)$.
		
		 With these compatible identifications, it follows from the discussion in \cite[2.2]{berger1994report} that $\lambda(\overline{R}/R)\geq \lambda(\overline{R}/\mathcal{D})$. In fact, the case where equality holds is referred to as \textit{maximal torsion}. In this case, the conjecture is proven to be true. We refer the reader to \cite[2.2]{berger1994report} for further details.
		
\end{enumerate}
\end{remark}

\begin{theorem}[Explicit Computation of $\hun$]\label{explicitcomp}
	Let	$\ds (R,\m ,k)$ be a Noetherian complete local one-dimensional domain such that $\ds R=S/I$ where $S=k[[X_1,\dots,X_n]]$, $k$ is algebraically closed of characteristic $0$, $n\geq 2$ and let $K$ be the fraction field of $R$. Assume $\ds I\subseteq (X_1,\dots,X_n)^2$. Suppose that $\overline{R}=k[[t]]$ with valuation function $v$.  
	Write $x_i(t)$ to denote the image of $X_i$ in $\overline{R}$ and set $x_i'(t)=\frac{d}{dt}x_i(t)$. Finally let $\mathcal{D}=Rx_1'(t)+\dots +Rx_n'(t)$ and $v(\mathcal{D}^{-1}):=\min\{v(\alpha)~|~\alpha\in K, \alpha \mathcal{D}\subseteq R\}$. Then we get $$\hun=\lambda(\overline{R}/\mathcal{D})-\lambda(\overline{R}/R)+v(\mathcal{D}^{-1})\leq v(\mathcal{D}^{-1}).$$
		If the rightmost inequality becomes an equality, then Berger's Conjecture is true.
\end{theorem}

\begin{proof} Since $\Rank_R(\Omega_{ R/k})=1$, we know that any ideal $J$, which is a surjective image of $\Omega_{ R/k}$, is isomorphic to $\Omega_{ R/k}/\tau(\Omega_{ R/k})$.
	 From \Cref{hcomprem1}(1), we have that $\mathcal{D}$ (considered as an $R$-submodule of $K$) is isomorphic to some ideal to which $\Omega_{ R/k}$ surjects. Hence, $\mathcal{D}\cong \Omega_{ R/k}/\tau(\Omega_{ R/k})$ and so  $\mathcal{D}$ is isomorphic to any ideal in $R$ to which $\Omega_{ R/k}$ surjects.
	 
	  Choose any such $J$ and let $J=\frac{a}{b}\mathcal{D}$ for some $a,b\in R$ (\Cref{rem3.1}). Using the inclusions $bJ\subseteq bR\subseteq R$, we obtain that $\lambda(R/bJ)=\lambda(R/J)+\lambda(R/bR)$. Similarly using $a\mathcal{D}\subseteq R\subseteq \overline{R}$, we get  $\lambda(R/a\mathcal{D})=\lambda(\overline{R}/a\mathcal{D})-\lambda(\overline{R}/R)$. Since $bJ=a\mathcal{D}$, we get $$\lambda(R/J)=\lambda(\overline{R}/a\mathcal{D})-\lambda(\overline{R}/R)-\lambda(R/bR)=\lambda(\overline{R}/\mathcal{D})+\lambda(\mathcal{D}/a\mathcal{D})-\lambda(\overline{R}/R)-\lambda(R/bR).$$ Since $\mathcal{D}$ is maximal Cohen-Macaulay of rank $1$ over $R$, we get $\lambda(\mathcal{D}/a\mathcal{D})=\lambda(R/aR)$ \cite[Corollary 4.6.11(c)]{bruns_herzog_1998} and hence we obtain that $$\lambda(R/J)=\lambda(\overline{R}/\mathcal{D})-\lambda(\overline{R}/R)+\lambda(R/aR)-\lambda(R/bR).$$ Using \Cref{prop.lem1}, we get that \begin{equation}\tag{\ref{explicitcomp}.1}\label{lengthcompeq}\lambda(R/J)=\lambda(\overline{R}/\mathcal{D})-\lambda(\overline{R}/R)+v(a)-v(b).\end{equation}
	By \Cref{hcomprem1}(2) we have that $\lambda(\overline{R}/\mathcal{D})\leq \lambda(\overline{R}/R)$ and so, taking the minimum over all such $J$'s, we obtain that $$\hun=\lambda(\overline{R}/\mathcal{D})-\lambda(\overline{R}/R)+v(D^{-1}).$$ The last statement again follows from \Cref{hcomprem1}(2).
\end{proof}

\begin{corollary}\label{explicitcompcor}
	Let	$\ds (R,\m ,k)$ be a Noetherian complete local one-dimensional domain such that $\ds R=S/I$ where $S=k[[X_1,\dots,X_n]]$, $k$ is algebraically closed of characteristic $0$, $n\geq 2$ and let $K$ be the fraction field of $R$. Assume $\ds I\subseteq (X_1,\dots,X_n)^{s+1}$ for $s\geq 1$. Suppose $\overline{R}=k[[t]]$ with valuation function $v$. Write $x_i(t)$ to denote the image of $X_i$ in $\overline{R}$ and set $x_i'(t)=\frac{d}{dt}x_i(t)$.  Finally let $\mathcal{D}=Rx_1'(t)+\dots +Rx_n'(t)$ and $v(\mathcal{D}^{-1})=\min\{v(\alpha)~|~\alpha\in K, \alpha \mathcal{D}\subseteq R\}$. If $$v(\mathcal{D}^{-1})< {n+s\choose s}\frac{s}{s+1}+\lambda(\overline{R}/R)-\lambda(\overline{R}/\mathcal{D}),$$\text{~then~}$\tau(\Omega_{R/k})\neq 0.$
\end{corollary}

\begin{proof}
	The proof is immediate from \Cref{mainthm} and \Cref{explicitcomp}.
\end{proof}

\subsection{Ways to make the above computations feasible on software}\label{hcomprem2}
	We specify two ways of making the above computation feasible on software. Recall that we are in the situation where $\overline{R}=k[[t]]$.
	
	\begin{enumerate}
		\item $\ds \lambda(\overline{R}/R)$ and $\lambda(\overline{R}/\mathcal{D})$ can be computed using \textit{suitable valuation semi-groups} -- significant discussions on this exist in the literature (see for instance, \cite[Proposition 2.9]{herzog1971wertehalbgruppe},\cite{yoshino1986torsion},\cite{MR1129516}) (namely, we need to count the missing valuations from the valuation semi-groups of $R$ and $\mathcal{D}$, respectively).
		
		\item Since $R$ is complete and $\overline{R}=k[[t]]$, the conductor $\cC=t^c \overline{R}$ where $c$ is the least integer such that $t^{c-1}\not \in R$ but $t^i\in R$ for all $i\geq c$. Hence $t^c\mathcal{D}$ is an ideal of $R$. Now substitute  $J=t^c\mathcal{D}$ in \Cref{lengthcompeq} to get $\lambda(R/t^c\mathcal{D})=\lambda(\overline{R}/\mathcal{D})-\lambda(\overline{R}/R)+c$. (Note that this calculation also shows that $\lambda({R}/t^c\mathcal{D})\leq c$ by \Cref{hcomprem1}(2).) 
		
		 So, from \Cref{lengthcompeq} and the above calculation we get that for any ideal $J=\frac{a}{b}\mathcal{D}$, $\lambda(R/J)=\lambda(R/t^c\mathcal{D})+(v(a)-v(b))-c.$ Hence, $$\hun=\lambda(R/t^c\mathcal{D})-c+v(\mathcal{D}^{-1}).$$

		\item Most of the time on Macaulay 2, the length is computable and so is the conductor valuation $c$. To compute $v(\mathcal{D}^{-1})$, one can proceed as follows:

		 -- By \Cref{trace_rem1}, $\text{tr}_R(\mathcal{D})=\text{tr}_R(\Omega_{R/k})$. 
		 
		 -- The latter can be computed using \Cref{vasconcelos}$(b)$. Then extract the least valuation (say $b$). 
		 
		 -- Note that $v(\mathcal{D})=v(x_i'(t))$ where $i$ is such that $x_i(t)$ has the least valuation. Using \Cref{vasconcelos}$(c)$, $v(\mathcal{D}^{-1})=b-v(\mathcal{D})$.

	\end{enumerate}

\vspace{1em}
The above discussions provide an efficient way of checking whether \Cref{mainthm} can be applied to a given instance of verifying Berger's Conjecture as well as simply computing $\hun$ or $v(\mathcal{D}^{-1})$. The following examples which were computed using Macaulay 2, illustrate this. 

\begin{example}\label{example2}
	$ R=\mathbb{C}[[t^8+t^9,64t^{10}-81t^{12},8t^{12}-9t^{13},t^{14},t^{15},t^{16},t^{17}]]={\mathbb{C}[[x,y,z,w,u,v,p]]}/{I}$ where $I$ is the defining ideal which has $21$ generators. Moreover,
	$\text{tr}_R(\Omega_{ R/k})=(p,v,u,w,z,y, x^2)$ and so
				$v(\text{tr}_R(\Omega_{ R/k}))=10$. Hence,  $v(\mathcal{D}^{-1})=10-7=3< \frac{7+1}{2}$. Hence, $\tau(\Omega_{R/k})\neq 0$ by \Cref{explicitcompcor}.
		
		In fact, here $c=14$ and Macaulay 2 computations on the ideal
		{\scriptsize $t^c\mathcal{D}=(z\,p,y\,p,w\,v,z\,v,71\,y\,v-2048\,x\,p,
		11360\,x\,u-11360\,x\,v-3237\,x\,p,355\,y^{2}-181760\,x\,z+2249280
		\,x\,w-2329155\,x\,v-597051\,x\,p)$} show that $\lambda(R/t^c\mathcal{D})=14$. Hence, $\hun=3$ (see \Cref{hcomprem2}(2)).

\end{example}

\begin{example}\label{example3}
Let $R=\mathbb{C}[[t^9,t^{14}+t^{15},t^{17},t^{29}]]=\mathbb{C}[[x,y,z,w]]/I$ where $I$ has $12$ generators. Here $\Omega_{ R/k}$  surjects to the ideal $$t^c\mathcal{D}=(x^{3}w,z^{4},15\,x\,y^{2}z-14\,x^{2}z^{2}+
14\,x\,y\,w-30\,x\,z\,w,x^{2}y^{2}-2\,x^{2}w-z\,w)$$ whose colength is $37$. Moreover, $\text{tr}_R(\Omega_{ R/k})=( w, z^2, yz, y^2+14xz, x^2z, x^2y, x^3+xz)$ whose valuation is $26$. Thus, $v(\mathcal{D}^{-1})=26-8=18$
and $\hun=37-40+18=15$. 
\end{example}

\begin{example}\label{example4}
	Let $R=\mathbb{C}[[t^4+t^5,t^9]]=\ds \frac{\mathbb{C}[[x,y]]}{(x^{9}-9\,x^{7}y+27\,x^{5}y^{2}-30\,x^{3}y^{3
	}+9\,x\,y^{4}-y^{5}-y^{4})}$. Further $\text{tr}_R(\Omega_{ R/k})=${\scriptsize $(51\,x^{4}+35\,x^{3}y+25\,x^{2}y^{2}+16\,x^{
3}-178\,x^{2}y-90\,x\,y^{2}-25\,y^{3}-12\,x\,y+76\,y^{2},125\,x^{3
}y^{2}+16\,x^{4}-109\,x^{3}y-20\,x^{2}y^{2}-250\,x\,y^{3}-28\,x^{2
}y+218\,x\,y^{2}+20\,y^{3}-4\,y^{2},25\,x^{4}y-16\,x^{4}+9\,x^{3}y
-55\,x^{2}y^{2}+28\,x^{2}y-18\,x\,y^{2}+5\,y^{3}+4\,y^{2},25\,x^{5
}+29\,x^{4}-55\,x^{3}y-16\,x^{3}-62\,x^{2}y+5\,x\,y^{2}+12\,x\,y+4
\,y^{2})$} and consequently we compute $v(\mathcal{D}^{-1})=12-3=9$. 

Finally note that $\lambda(\overline{R}/\mathcal{D})=9$ (the valuations missing from the value semi-group of $\mathcal{D}$ are $\{0,1,2,4,5,6,9,10,14\}$) and $\lambda(\overline{R}/R)=12$ (missing valuations from the semi-group of $R$ are $\{1,2,3,5,6,7,10,11,14,15,19,23\}$.) Thus from \Cref{explicitcomp}, we get that $\hun=6$.
\end{example}

\section*{Conclusion}\label{conclusion} We conclude with the following observations and questions. 

\begin{enumerate}

	\item \Cref{prop.propo1} can be used to conclude that in a one-dimensional complete local domain, any non-zero reflexive ideal is isomorphic to an ideal containing the conductor. For a more general statement, see \cite[Theorem 3.5]{dao2021reflexive}.
	
		\item It is natural to ask if $(a)$ implies $(b)$ in \Cref{prop.propo1} without the $DVR$ assumption.
	
	\item If $R$ is quasi-homogeneous, then  $\text{tr}_R(\Omega_{ R/k})=\m$. 
	\begin{proof} Since $R$ is not regular, $\m=\text{tr}_R(\m)$. Since $R$ is quasi-homogeneous, we are done by \Cref{trace_rem1}.
	\end{proof} 
It is natural to ask if this condition is also sufficient.

	\item $\ds \h(\omega)$ itself can be an interesting object of study. We know that it is $0$ if and only if $R$ is Gorenstein. Since it is an invariant of the ring, a further study of $\h(\omega)$ may help in classification problems.

\end{enumerate}

\section*{Acknowledgement} 
I am highly grateful to Prof. Craig Huneke for extremely helpful conversations regarding the ideas and proofs in this paper. I am also indebted to Vivek Mukundan for not only helping me write Macaulay2 codes to check examples but also for valuable conversations. I would like to thank Luis N\'{u}\~{n}ez-Betancourt for pointing out the references of Teissier and Hironaka. I am thankful to Prof. Charles Weibel for his helpful comments on this article. I am also immensely grateful to the anonymous referee for their suggestions in improving the exposition of this article.


\end{document}